\documentclass[12pt,3p]{nj_style}

\usepackage{amsmath}
\usepackage{amssymb}
\usepackage{amsthm}
\usepackage{enumerate}
\usepackage[font=small,labelfont=bf,margin=8pt]{caption}
\usepackage[usenames,dvipsnames]{color}
\usepackage{graphicx}
\usepackage[colorlinks=true]{hyperref}
\usepackage{bbm}
 
\newtheorem{thm}{Theorem} 

\newtheorem{lemma}[thm]{Lemma}
\newtheorem{prop}[thm]{Proposition}

\newtheorem{defn}[thm]{Definition}

\newcommand{\Tr}{{\rm Tr}}
\newcommand{\bb}[1]{\mathbb{#1}}
\newcommand{\cl}[1]{\mathcal{#1}}

\newcommand{\Ad}[1]{{\rm Ad}_{#1}}
\newcommand{\Pmaps}{\mathcal{P}}
\newcommand{\Pk}[1]{\Pmaps_{#1}}
\newcommand{\SP}{\mathcal{SP}}
\newcommand{\SPk}[1]{\SP_{#1}}
\newcommand{\CP}{\mathcal{CP}}

\begin{document}

\begin{frontmatter}


\title{Generation of Mapping Cones from Small Sets}
\author{Nathaniel Johnston}
\ead{njohns01@uoguelph.ca}
\address{Department of Mathematics and Statistics, University of Guelph, Guelph, Ontario N1G~2W1, Canada}

\author{{\L}ukasz Skowronek}
\ead{lukasz.skowronek@uj.edu.pl}
\address{Instytut Fizyki im. Smoluchowskiego, Uniwersytet Jagiello\'{n}ski, Reymonta 4, 30-059 Krak\'{o}w, Poland}

\author{Erling St{\o}rmer}
\ead{erlings@math.uio.no}
\address{Department of Mathematics, University of Oslo, P.O. Box 1053 Blindern, NO-0316 Oslo, Norway}

\begin{abstract}
We answer in the affirmative a recently-posed question that asked if there exists an ``untypical'' convex mapping cone -- i.e., one that does not arise from the transpose map and the cones of $k$-positive and $k$-superpositive maps. We explicitly construct such a cone based on atomic positive maps. Our general technique is to consider the smallest convex mapping cone generated by a single map, and we derive several results on such mapping cones. We use this technique to also present several other examples of untypical mapping cones, including a family of cones generated by spin factors. We also provide a full characterization of mapping cones generated by single elements in the qubit case in terms of their typicality.
\end{abstract}


\end{frontmatter}

\section{Introduction}

As an attempt to classify positive maps between operator algebras, one of us introduced the concept of \emph{mapping cones} \cite{S86}, which in the finite-dimensional case are closed cones of positive maps that are closed under the composition with completely positive maps. Very few examples of mapping cones have so far been exhibited, namely those which arise naturally from completely positive maps, $k$-positive maps and $k$-superpositive maps and their composition with the transpose map. In \cite{Sko11}, one of us posed the problem of whether there exist other mapping cones than those described above. In the present paper we answer this question in the affirmative as we exhibit several other examples of mapping cones, thus showing that the theory of positive maps is very complicated, even in the case of $3 \times 3$ matrices.

Our main approach to this problem is to consider convex mapping cones that are generated by small sets of positive maps -- that is, the smallest convex mapping cone that contains a given set of maps. We show that in many cases, the mapping cone generated by a single map gives a well-known typical mapping cone, such as the cone of completely positive maps or the cones of $k$-superpositive maps. We then provide several examples to show that untypical mapping cones can also arise naturally in this way. In particular, we show that the convex mapping cone generated by an atomic map is always untypical, we provide examples of non-atomic maps that generate untypical mapping cones, and we completely characterize whether such cones are typical or untypical in the case of $2 \times 2$ matrices.

These results are of interest in quantum entanglement theory, where mapping cones have been shown to play an important role \cite{SSZ09,SS10}, as the cones of $k$-positive, $k$-superpositive, completely positive, and completely co-positive maps all arise frequently in this setting \cite{HHH96,P96,HSR03,CK06,HHH09}. We consider a natural partial order that arises from our method of generating mapping cones, and briefly consider its implications in entanglement theory. In particular, we show that it is related to the notion of optimality of entanglement witnesses introduced in \cite{LKCH00}.

In Section~\ref{sec:prelims} we introduce our notation and the typical mapping cones of $k$-positive and $k$-superpositive maps. In Section~\ref{sec:cone_generated} we show how to construct a minimal convex mapping cone containing a given set of positive maps, and use this construction to build untypical convex mapping cones in Section~\ref{sec:untypical_cones}. We present another in-depth example in Section~\ref{sec:spin_factors}, where we construct a family of untypical mapping cones based on spin factors. We show that these mapping cones are analogous to the cones of $k$-superpositive maps in a natural way. We close in Section~\ref{sec:partial_order} by considering a partial order based on the generation of mapping cones that measures how well one set of positive maps detects entanglement compared to another set.

\section{Notation and Preliminaries}\label{sec:prelims}

We use $\cl{H}$ to denote a finite-dimensional Hilbert space and $\cl{L}(\cl{H})$ the space of linear maps on $\cl{H}$. If we wish to emphasize the dimension $n$ of a Hilbert space then we will write it as $\cl{H}_n$. If $X \in \cl{L}(\cl{H})$ is positive then we write $X \geq 0$.

Many of the cones of operators and linear maps that we deal with will be inspired by the cone of separable operators, which is particularly important in quantum information theory. An operator $0 \leq X \in \cl{L}(\cl{H}) \otimes \cl{L}(\cl{H})$ is called \emph{separable} if it can be written in the form
\begin{align*}
	X = \sum_i Y_i \otimes Z_i \quad \text{ with } \quad Y_i, Z_i \geq 0 \quad \forall \, i.
\end{align*}
Note that without loss of generality we can choose each $Y_i$ and $Z_i$ to have rank one. More generally, we say that the \emph{Schmidt number} \cite{TH00} of an operator $0 \leq X \in \cl{L}(\cl{H}) \otimes \cl{L}(\cl{H})$ (denoted by $SN(X)$) is the smallest integer $k$ so that we can write $X = \sum_i \mathbf{v}_i\mathbf{v}_i^*$ (here ${\mathbf v_i}^*$ is the dual vector of ${\mathbf v_i}$ and ${\mathbf v_i}{\mathbf v_i}^*$ is the outer product of ${\mathbf v_i}$ with itself), where each $\mathbf{v}_i$ can be written in the form $\mathbf{v}_i = \sum_{j=1}^k \mathbf{w}_{ij} \otimes \mathbf{z}_{ij}$. It is straightforward to verify that $SN(X) = 1$ if and only if $X$ is separable.

\subsection{Cones of Positive Maps}\label{sec:positive_maps}

A map $\Phi : \cl{L}(\cl{H}) \rightarrow \cl{L}(\cl{H})$ is said to be \emph{positive} if $\Phi(X) \geq 0$ whenever $X \geq 0$. Similarly, $\Phi$ is called \emph{$k$-positive} if $id_k \otimes \Phi$ is positive, where $id_k$ denotes the identity map on $\cl{L}(\cl{H}_k)$, and $\Phi$ is called \emph{completely positive} if $\Phi$ is $k$-positive for all $k \in \bb{N}$. Also, $\Phi$ is called \emph{$k$-copositive} if $\Phi$ is $k$-positive, where $t$ is the transpose map on $\cl{L}(\cl{H})$. We will use $\cl{P}_k(\cl{L}(\cl{H}))$ and $\cl{CP}(\cl{L}(\cl{H}))$ to denote the sets of $k$-positive and completely positive maps on $\cl{L}(\cl{H})$ respectively. We may abbreviate this notation as simply $\cl{P}_k$ or $\cl{CP}$ when the Hilbert space the maps act on is understood or unimportant. Note that $\cl{P}_k$ and $\cl{CP}$ are cones (i.e., they are closed under multiplication by non-negative scalars), closed, and convex.

Given an operator $A \in \cl{L}(\cl{H})$, we define the \emph{adjoint map} ${\rm Ad}_A : \cl{L}(\cl{H}) \rightarrow \cl{L}(\cl{H})$ by ${\rm Ad}_A(X) = AXA^*$, where $A^*$ is the Hermitian adjoint of $A$. It is clear that ${\rm Ad}_A$ is always completely positive. A well-known characterization of completely positive maps \cite{C75} says that $\cl{CP}(\cl{L}(\cl{H}_n)) = \cl{P}_n(\cl{L}(\cl{H}_n))$ and furthermore that $\Phi$ is completely positive if and only if there exist operators $A_i \in \cl{L}(\cl{H})$ such that $\Phi = \sum_i {\rm Ad}_{A_i}$. In other words, the adjoint maps are the extreme points of the set of completely positive maps.

Given a fixed orthonormal basis $\{{\mathbf e_i}\}_{i=1}^n$ of $\cl{H}_n$, the \emph{Jamio{\l}kowski--Choi isomorphism} \cite{C75,J72} associates a linear map $\Phi : \cl{L}(\cl{H}_n) \rightarrow \cl{L}(\cl{H}_n)$ with the operator $C_{\Phi} := \sum_{i,j=1}^n {\mathbf e_i}{\mathbf e_j}^* \otimes \Phi({\mathbf e_i}{\mathbf e_j}^*) \in \cl{L}(\cl{H}_n) \otimes \cl{L}(\cl{H}_n)$. The operator $C_{\Phi}$ is called the \emph{Choi matrix} of $\Phi$. For us, it will be useful to know that $\Phi$ is completely positive if and only if $C_{\Phi}$ is positive.

In the case when we can write $\Phi = \sum_i {\rm Ad}_{A_i}$ with ${\rm rank}(A_i) \leq k$ for all $i$, $\Phi$ is called \emph{$k$-superpositive} \cite{SSZ09} (or simply \emph{superpositive} \cite{And04} in the $k = 1$ case), and we denote these cones by $\SPk{k}(\cl{L}(\cl{H}))$ or simply $\SPk{k}$. In quantum information theory, superpositive maps are usually called \emph{entanglement-breaking maps} \cite{HSR03} because they are exactly the maps with the property that $(id \otimes \Phi)(X)$ is separable for all $X \geq 0$. More generally, $\Phi$ is $k$-superpositive if and only if $SN((id \otimes \Phi)(X)) \leq k$ for all $X \geq 0$, if and only if $SN(C_\Phi) \leq k$ \cite{CK06}.

Given a cone of positive maps $\cl{M} \subseteq \cl{P}_1$, we define the cone of Choi matrices $C_{\cl{M}} := \{ C_{\Phi} : \Phi \in \cl{M} \}$ and the cone of dual maps $\cl{M}^{\dagger} := \{ \Phi^{\dagger} : \Phi \in \cl{C} \}$, where $\Phi^{\dagger} : \cl{L}(\cl{H}) \rightarrow \cl{L}(\cl{H})$ is the unique map defined via the Hilbert-Schmidt inner product so that $\Tr(\Phi(X)Y) = \Tr(X\Phi^{\dagger}(Y))$ for all $X,Y \in \cl{L}(\cl{H})$.

\subsection{Mapping Cones}\label{sec:mapping_cone}

A \emph{mapping cone} \cite{S86} is a nonzero closed cone $\cl{M} \subseteq \cl{P}_1$ with the property that $\Phi \circ \Omega \circ \Psi \in \cl{M}$ whenever $\Omega \in \cl{M}$ and $\Phi,\Psi \in \cl{CP}$. The cones $\cl{P}_k$ of $k$-positive maps and $\SPk{k}$ of $k$-superpositive maps are the prototypical examples of mapping cones and can be seen repeatedly in recent work on mapping cones \cite{JS11,Sko11,SSZ09,SS10,S11}. Other well-known examples of mapping cones include those of the form $\cl{M} \circ t := \{\Phi \circ t : \Phi \in \cl{M}\}$, where $\cl{M}$ is equal to either $\cl{P}_k$ or $\SPk{k}$. Furthermore, the intersection $\cl{M}_1 \cap \cl{M}_2$ of two mapping cones $\cl{M}_1$ and $\cl{M}_2$ is again a mapping cone, as is the sum $\cl{M}_1 \vee \cl{M}_2 := \{ \Phi + \Psi : \Phi \in \cl{M}_1, \Psi \in \cl{M}_2 \}$.

In \cite{Sko11} it was noted that all convex mapping cones that have been considered in the past can be constructed via the methods described in the previous paragraph. Hence any mapping cone arising from $\cl{P}_k$ or $\SPk{k}$ via the operations $\cl{M} \mapsto \cl{M} \circ t$, $(\cl{M}_1,\cl{M}_2) \mapsto \cl{M}_1 \cap \cl{M}_2$, or $(\cl{M}_1,\cl{M}_2) \mapsto \cl{M}_1 \vee \cl{M}_2$ was called \emph{typical}, and it was asked whether or not there exist convex mapping cones that are \emph{untypical}. We construct such convex mapping cones in Sections~\ref{sec:untypical_cones} and~\ref{sec:spin_factors}, which shows that they really do provide a non-trivial generalization of the cones of $k$-positive and $k$-superpositive maps.

\section{Convex Mapping Cones Generated by Small Sets of Maps}\label{sec:cone_generated}

In a sense, there is nothing particularly special about the transpose map $t$ and its appearance in the definition of a typical convex mapping cone. The cone $\cl{CP} \circ t$ is the smallest convex mapping cone containing $t$, but there is no reason that we can't similarly define the smallest convex mapping cone containing any other given set of positive maps. Indeed, for any set of positive maps $\cl{Q} \subset \cl{P}_1$, we define the \emph{convex mapping cone generated by $\cl{Q}$} as follows:
\begin{align*}
	\cl{M}_{\cl{Q}} := \left\{ \sum_i \Phi_i \circ \Omega_i \circ \Psi_i : \Omega_i \in \cl{Q}, \Phi_i, \Psi_i \in \cl{CP} \ \forall \, i \right\}.
\end{align*}

It is clear that $\cl{M}_{\cl{Q}}$ is convex and a mapping cone, and furthermore that it is the smallest convex mapping cone such that contains $\cl{Q}$. That is, if $\cl{M}$ is a convex mapping cone such that $\cl{M} \supseteq \cl{Q}$, then $\cl{M} \supseteq \cl{M}_{\cl{Q}}$. Furthermore, for any sets $\cl{Q}_i$ we have $\cl{M}_{\cup_i \cl{Q}_i} = \vee_i \cl{M}_{\cl{Q}_i}$. Thus it is of particular interest to understand the convex mapping cones $\cl{M}_{\cl{Q}}$, where $\cl{Q}$ is a singleton set, since all convex mapping cones can be obtained by adding these cones together.

Note that the mapping cone generated by $\cl{Q}$ has a natural interpretation in quantum information theory if we use the Jamio{\l}kowski--Choi isomorphism. For example, in the simplest case of a singleton set, the set of Choi matrices of maps in $\cl{M}_{\{\Phi\}}$ is exactly the set $\{\sum_i {\rm Ad}_{A_i \otimes B_i}(C_\Phi)\}$. That is, it is the set of operators that can reached from $C_\Phi$ via maps of the form $\sum_i {\rm Ad}_{A_i \otimes B_i}$, which are called \emph{separable maps} \cite{CDKL01,Rai97}. The problem of determining what types of operators can be reached from a given operator by applying separable maps is a \emph{distillation problem}. For example, it is often asked whether a given operator can be distilled via a separable map into the ``maximally-entangled'' operator $\sum_{i,j=1}^n {\mathbf e_i}{\mathbf e_j}^* \otimes {\mathbf e_i}{\mathbf e_j}^*$ \cite{HHH98}. Equivalently, this is the problem of determining whether or not $id \in \cl{M}_{\{\Phi\}}$ (and hence $\cl{CP} \subseteq \cl{M}_{\{\Phi\}}$).

We now present some special cases of singleton sets $\cl{Q}$ that generate well-known mapping cones.
\begin{prop}\label{prop:sing_map_cone_id}
	$\cl{M}_{\{id\}} = \cl{CP}$.
\end{prop}
\begin{proof}
	Trivial, as the composition and sum of completely positive maps is again completely positive.
\end{proof}

\begin{prop}\label{prop:sing_map_cone_trans}
	$\cl{M}_{\{t\}} = \cl{CP} \circ t$.
\end{prop}
\begin{proof}
	If $\Phi,\Psi \in \cl{CP}$ then $\Phi \circ t \circ \Psi = \big(\Phi \circ (t \circ \Psi \circ t)\big) \circ t$. Because $t \circ \Psi \circ t \in \cl{CP}$, we have $\Phi \circ t \circ \Psi \in \cl{CP} \circ t$.
\end{proof}

\begin{prop}\label{prop:sing_map_cone_superpos}
	If $\Omega \in \SPk{1}$, then $\cl{M}_{\{\Omega\}} = \SPk{1}$.
\end{prop}
\begin{proof}
	Because $\SPk{1}$ is a convex mapping cone, it follows that $\cl{M}_{\{\Omega\}} \subseteq \SPk{1}$ by the fact that $\cl{M}_{\{\Omega\}}$ is the smallest convex mapping cone containing $\Omega$.	To see the other inclusion, recall \cite[Lemma~2.4]{S86} that if $\cl{M}$ is any convex mapping cone then $\cl{M} \supseteq \SPk{1}$, so $\cl{M}_{\{\Omega\}} \supseteq \SPk{1}$.
\end{proof}

Our next result of this type concerns the \emph{reduction map} $R : \cl{L}(\cl{H}_n) \rightarrow \cl{L}(\cl{H}_n)$ defined as follows:
	\begin{align*}
		R(X) = \Tr(X)I - X.
	\end{align*}
	The map $R$ is clearly positive because $\Tr(X) \geq \|X\|$ whenever $X \geq 0$. However, it is easily-verified that $R$ is not completely positive (or even $2$-positive \cite{Tom85}). The positivity properties of $R$ have led to it playing an important role in quantum information theory \cite{CAG99,HH99}.
\begin{prop}\label{prop:sing_map_cone_reduc}
	$\cl{M}_{\{R\}} = \SPk{2} \circ t$ (and equivalently, $\cl{M}_{\{R \circ t\}} = \SPk{2}$).
\end{prop}
\begin{proof}
	To see that $\cl{M}_{\{R \circ t\}} \subseteq \SPk{2}$, it is enough to show that $R \circ t \in \SPk{2}$. To this end, we consider its Choi matrix:
	\begin{align*}
		C_{R \circ t} = I - \sum_{i,j=1}^n \mathbf{e}_i\mathbf{e}_j^* \otimes \mathbf{e}_j\mathbf{e}_i^* = \sum_{i>j=1}^n \big(\mathbf{e}_i \otimes \mathbf{e}_j - \mathbf{e}_j \otimes \mathbf{e}_i\big)\big(\mathbf{e}_i \otimes \mathbf{e}_j - \mathbf{e}_j \otimes \mathbf{e}_i\big)^*,
	\end{align*}
	which evidently has Schmidt number no larger than $2$. It follows that $\cl{M}_{\{R \circ t\}} \subseteq \SPk{2}$.
	
	To see that $\SPk{2} \subseteq \cl{M}_{\{R \circ t\}}$, suppose that $\Phi \in \SPk{2}$. Then $SN(C_\Phi) \leq 2$, so there are families of vectors $\{\mathbf{a}_k\},\{\mathbf{b}_k\},\{\mathbf{c}_k\},\{\mathbf{d}_k\} \subseteq \cl{H}_n$ such that
	\begin{align*}
		C_{\Phi} = \sum_{k} \big(\mathbf{a}_k \otimes \mathbf{b}_k + \mathbf{c}_k \otimes \mathbf{d}_k\big)\big(\mathbf{a}_k \otimes \mathbf{b}_k + \mathbf{c}_k \otimes \mathbf{d}_k\big)^*.
	\end{align*}
	Let $A_k \in \cl{L}(\cl{H}_n)$ be the operator defined by $A_k^t\mathbf{e}_1 = \mathbf{a}_k$, $A_k^t\mathbf{e}_2 = \mathbf{c}_k$, and $A_k^t\mathbf{e}_j = 0$ for $j \geq 3$. Let $B_k \in \cl{L}(\cl{H}_n)$ be the operator defined by $B_k\mathbf{e}_1 = \mathbf{d}_k$, $B_k\mathbf{e}_2 = -\mathbf{b}_k$, and $B_k\mathbf{e}_j = 0$ for $j \geq 3$. Then
	\begin{align*}
		C_{{\rm Ad}_{A_k} \circ (R \circ t) \circ {\rm Ad}_{B_k}} & = (A_k^t \otimes B_k)C_{R \circ t}(A_k^t \otimes B_k)^* \\
		& = (A_k^t \otimes B_k)\left(\sum_{i>j=1}^n \big(\mathbf{e}_i \otimes \mathbf{e}_j - \mathbf{e}_j \otimes \mathbf{e}_i\big)\big(\mathbf{e}_i \otimes \mathbf{e}_j - \mathbf{e}_j \otimes \mathbf{e}_i\big)^*\right)(A_k^t \otimes B_k)^* \\
		& = \big(\mathbf{a}_k \otimes \mathbf{b}_k + \mathbf{c}_k \otimes \mathbf{d}_k\big)\big(\mathbf{a}_k \otimes \mathbf{b}_k + \mathbf{c}_k \otimes \mathbf{d}_k\big)^*.
	\end{align*}
	It follows that $\Phi = \sum_k {\rm Ad}_{A_k} \circ (R \circ t) \circ {\rm Ad}_{B_k}$, so $\Phi \in \cl{M}_{\{R \circ t\}}$ and $\SPk{2} \subseteq \cl{M}_{\{R \circ t\}}$.
\end{proof}

Propositions~\ref{prop:sing_map_cone_id}, \ref{prop:sing_map_cone_superpos}, and~\ref{prop:sing_map_cone_reduc} show that there are singleton sets $\cl{Q}$ such that $\cl{M}_{\cl{Q}} = \SPk{k}$ for $k \in \{1,2,n\}$. Our final result of this type demonstrates that there is in fact a singleton set $\cl{Q}$ such that $\cl{M}_{\cl{Q}} = \SPk{k}$ for any $1 \leq k \leq n$.
\begin{prop}\label{prop:sing_map_cone_ksuper}
	Let $E_k \in \cl{L}(\cl{H}_n)$ be any operator with rank $k$. Then $\cl{M}_{\{{\rm Ad}_{E_k}\}} = \SPk{k}$.
\end{prop}
\begin{proof}
	Because ${\rm rank}(E_k) = k$, we have ${\rm Ad}_{E_k} \in \SPk{k}$, so $\cl{M}_{\{{\rm Ad}_{E_k}\}} \subseteq \SPk{k}$. To see the other inclusion, consider an arbitrary map $\Phi \in \SPk{k}$, written in the form $\Phi = \sum_i {\rm Ad}_{A_i}$ with ${\rm rank}(A_i) \leq k$ for all $i$. Because of this rank condition, there exist operators $\{B_i\}$ and $\{C_i\}$ so that $A_i = B_i E_k C_i$ for all $i$. Then $\Phi = \sum_i {\rm Ad}_{B_i} \circ {\rm Ad}_{E_k} \circ {\rm Ad}_{C_i}$, so $\Phi \in \cl{M}_{\{{\rm Ad}_{E_k}\}}$ and $\SPk{k} \subseteq \cl{M}_{\{{\rm Ad}_{E_k}\}}$.
\end{proof}

We now have seen that there are singleton sets that give rise to the cones $\SPk{k}$ (and also $\SPk{k} \circ t$) for any $1 \leq k \leq n$. The analogous problem of generating $\cl{P}_k$ and $\cl{P}_k \circ t$, however, is much more difficult. In the $n = 2$ case, we have $\cl{P}_1 = \cl{CP} \vee \cl{CP} \circ t$, so $\cl{P}_1 = \cl{P}_1 \circ t = \cl{M}_{\{id,t\}}$. However, when $n \geq 3$ it is the case that $\cl{P}_1$ is not generated by any finite (or even countable) set of maps \cite{Sko12}. In general we are not aware of an answer to the question of whether or not there exists a finite set $\cl{Q}_k$ such that $\cl{M}_{\cl{Q}_k} = \cl{P}_k$ for $1 < k < n$.

\section{Examples of Untypical Mapping Cones}\label{sec:untypical_cones}

In the previous section we showed that the mapping cone generated by a single map, in many cases, gives a well-known typical mapping cone. In contrast, we now present several examples of maps that generate mapping cones that are untypical. In particular, we show that all atomic maps lead to untypical mapping cones, yet there are many non-atomic maps that also give rise to untypical mapping cones. We also completely characterize (un)typicality of mapping cones generated in this way in the $n = 2$ case.

\subsection{Atomic Maps Generate Untypical Mapping Cones}\label{sec:atomic_map}

We now specialize to the case when $\cl{Q} = \{\Phi\}$, where $\Phi : \cl{L}(\cl{H}_n) \rightarrow \cl{L}(\cl{H}_n)$ is an \emph{atomic map} -- that is, a positive map that can not be written as a sum of a $2$-positive and a $2$-copositive map. Atomic maps exist exactly when $n \geq 3$, and the most famous example is the \emph{Choi map} \cite{Cho75} defined by
\begin{align*}
	(x_{ij}) \mapsto \begin{bmatrix}x_{11} + x_{33} & -x_{12} & -x_{13} \\ -x_{21} & x_{11} + x_{22} & -x_{23} \\ -x_{31} & -x_{32} & x_{22} + x_{33}\end{bmatrix}.
\end{align*}

There has been much work done recently to construct positive atomic maps when $n = 3$ \cite{TT88,CKL92,Osa91,BFP04b,Hal06}. We now show that any such map generates an untypical mapping cone.

\begin{thm}\label{thm:atomic_untypical}
	Let $\Phi : \cl{L}(\cl{H}_n) \rightarrow \cl{L}(\cl{H}_n)$ be an atomic positive map. Then $\cl{M}_{\{\Phi\}}$ is untypical.
\end{thm}
\begin{proof}
	The only typical mapping cone not contained in $\cl{P}_2 \vee (\cl{P}_2 \circ t)$ is $\cl{P}$ itself. Since $\Phi \notin \cl{P}_2 \vee (\cl{P}_2 \circ t)$, we only need to show that $\cl{M}_{\{\Phi\}} \neq \cl{P}_1$. To this end, we show that $id \notin \cl{M}_{\{\Phi\}}$. To see why this claim holds, note that $id$ is extreme in the set of positive maps, so if $id \in \cl{M}_{\{\Phi\}}$ there must exist $A,B \in \cl{L}(\cl{H}_n)$ such that $id = {\rm Ad}_A \circ \Phi \circ {\rm Ad}_B$. Since $id$ has full rank as a linear operator, each of $A$ and $B$ must be invertible, so ${\rm Ad}_{A^{-1}} \circ id \circ {\rm Ad}_{B^{-1}} = \Phi$, which is completely positive. Since $\Phi$ is not completely positive, this is a contradiction, so $\cl{M}_{\{\Phi\}}$ must be untypical.
\end{proof}

Note that the results of \cite{Mar10} imply that an extreme positive map is either of the form ${\rm Ad}_A$, ${\rm Ad}_A\circ t$, or it is atomic. Therefore, Theorem \ref{thm:atomic_untypical} applies to all extreme positive maps that are not of the form ${\rm Ad}_A$ or ${\rm Ad}_A\circ t$. 

\subsection{Untypical Mapping Cones Arising from Non-Atomic Positive Maps}\label{sec:non_atomic_map}

As another example of how untypical mapping cones can arise, we present the following (slightly technical) theorem, which is proved via a series of lemmas throughout this section. Note that we use ${\rm supp}(U)$ to denote the support of the operator $U \in \cl{L}(\cl{H}_n)$.
\begin{thm}\label{thm:untypical_from_non_atomic}
	Let $U,V \in \cl{L}(\cl{H}_3)$ have ${\rm rank}(U) = {\rm rank}(V) = 2$. Assume that ${\rm range}(U) = {\rm supp}(U) = {\rm supp}(V)$, ${\rm range}(U)$ and ${\rm range}(V)$ commute, and ${\rm range}(U) \cap {\rm range}(V)$ has dimension $1$. If $\Phi := \Ad{U}+\Ad{V} \notin \cl{P}_2 \circ t$ then $\cl{M}_{\{\Phi\}}$ is untypical.
\end{thm}
Before proving Theorem~\ref{thm:untypical_from_non_atomic}, we note that the following operators $U,V \in \cl{L}(\cl{H}_3)$ satisfy all of its hypotheses and thus provide a concrete example of such an untypical mapping cone:
\begin{equation}\label{matrUV}
U:=\begin{bmatrix}1&0&0\\0&1&0\\0&0&0\end{bmatrix},\quad V:=\begin{bmatrix}0&1&0\\0&0&0\\1&0&0\end{bmatrix}.
\end{equation}
The corresponding map $\Phi:=\Ad{U}+\Ad{V}$ acts as follows:
\begin{align*}
	(x_{ij}) \mapsto \begin{bmatrix}x_{11} + x_{22} & x_{12} & x_{21} \\ x_{21} & x_{22} & 0 \\ x_{12} & 0 & x_{11}\end{bmatrix}.
\end{align*}
The only slightly non-trivial property of $U$, $V$, and $\Phi$ that needs to be checked is that $\Phi \notin \cl{P}_2 \circ t$. To see this, one can verify that
\begin{equation}
\left(id_2 \otimes \Phi\circ t\right)\left(\begin{bmatrix}1&0&0&0&1&0\\0&0&0&0&0&0\\0&0&0&0&0&0\\0&0&0&0&0&0\\1&0&0&0&1&0\\0&0&0&0&0&0\end{bmatrix}\right)=\begin{bmatrix}1&0&0&0&0&1\\0&0&0&1&0&0\\0&0&1&0&0&0\\0&1&0&1&0&0\\0&0&0&0&1&0\\1&0&0&0&0&0\end{bmatrix}\not\geq 0,
\end{equation}
so $\Phi$ is not an element of $\Pk{2}\circ t$, as desired.

We now prove Theorem~\ref{thm:untypical_from_non_atomic} via a series of lemmas.
\begin{lemma}\label{lem:ad_aub_avb}
	Under the hypotheses of Theorem~\ref{thm:untypical_from_non_atomic}, there do not exist $\lambda \in \bb{C}$ and $A,B \in \cl{L}(\cl{H}_3)$ such that $AUB = U$ and $AVB = \lambda U$.
\end{lemma}
\begin{proof}
	Since $U = AUB$, ${\rm range}(B) \supseteq {\rm supp}(U)$ and ${\rm supp}(B) \supseteq {\rm supp}(U)$. The same formulas hold if $B$ is replaced by ${\rm supp}(U)B{\rm supp}(U)$ which equals ${\rm supp}(V)B{\rm supp}(U)$ by hypothesis. Thus
	\begin{align}\label{eq:vb_range_vu}
		{\rm range}(VB) = {\rm range}(V{\rm range}(V)B{\rm supp}(U)) = {\rm range}(V{\rm supp}(U)) = {\rm range}(VU).
	\end{align}
	We can thus replace $A$ by $A{\rm range}(U)$ and thus assume $A$ is invertible.
	
	Then $UB = A^{-1}U$ and $VB = \lambda A^{-1}U$, so $VB = \lambda UB$. By Equation~\eqref{eq:vb_range_vu} it follows that ${\rm range}(VB) \subseteq {\rm range}(U){\rm range}(V)$. Then ${\rm rank}(AVB) \leq 1$, which contradicts the fact that ${\rm rank}(U) = 2$, except when $\lambda = 0$. In this case, $VB = 0$. Thus ${\rm range}(B) \subseteq {\rm ker}(V)$, which has dimension $1$, by hypothesis. Then ${\rm rank}(B) \leq 1$, which is impossible since $AUB = U$, which has rank $2$.
\end{proof}

\begin{lemma}\label{lem:adu_notin_mphi}
	Under the hypotheses of Theorem~\ref{thm:untypical_from_non_atomic}, $\Ad{U} \notin \cl{M}_{\{\Phi\}}$.
\end{lemma}
\begin{proof}
	Assume that $\Ad{U} \in \cl{M}_{\{\Phi\}}$. Because $\Ad{U}$ is an extremal element of the cone $\cl{P}_1$, it is also extremal in $\cl{M}_{\{\Phi\}}$. Hence there exist $A,B \in \cl{L}(\cl{H}_n)$ such that
	\begin{align*}
		\Ad{U} = \Ad{A} \circ \Ad{U} \circ \Ad{B} + \Ad{A} \circ \Ad{V} \circ \Ad{B} = \Ad{AUB} + \Ad{AVB}.
	\end{align*}
	By using extremality of $\Ad{U}$ in $\cl{P}_1$ again, we see that there exist $\alpha,\beta > 0$ such that
	\begin{align*}
		\Ad{AUB} = \alpha \Ad{U} \text{ and } \Ad{AVB} = \beta \Ad{U}.
	\end{align*}
	It follows that there exist $x,y \in \bb{C}$ with $|x| = |y| = 1$ such that $AUB = x\sqrt{\alpha} U$ and $AVB = y\sqrt{\beta} U$. By absorbing the constant $1/(x\sqrt{\alpha})$ into $A$ and defining $\lambda := y\sqrt{\beta}/(x\sqrt{\alpha})$, we can reduce this system of equalities slightly to $AUB = U$, $AVB = \lambda U$. By Lemma~\ref{lem:ad_aub_avb}, this gives a contradiction and proves the result.
\end{proof}

It is worth observing that Lemmas~\ref{lem:ad_aub_avb} and~\ref{lem:adu_notin_mphi} generalize slightly to the case where $U,V \in \cl{L}(\cl{H}_n)$ have ${\rm rank}(U) = {\rm rank}(V) = k > n/2$ and ${\rm range}(U) \cap {\rm range}(V)$ has dimension at most $k-1$. The lemmas as-stated arise in the $n = 3$, $k = 2$ case, which is the case of interest to us.
\begin{lemma}\label{lem:mphi_props}
	Under the hypotheses of Theorem~\ref{thm:untypical_from_non_atomic}, we have the following:
	\begin{enumerate}[(1)]
		\item $\cl{M}_{\{\Phi\}}\not\subseteq\Pk{2}\circ t$, and
		\item $\SPk{2}\not\subseteq \cl{M}_{\{\Phi\}}$.
	\end{enumerate}
\end{lemma}
\begin{proof}
 The first fact is true because $\Phi\not\in\Pk{2}\circ t$ (by hypothesis). The second fact is a consequence of Lemma~\ref{lem:adu_notin_mphi} and $\Ad{U}\in\SPk{2}$.
\end{proof}
To show that $\cl{M}_{\{\Phi\}}$ is not typical, it is now sufficient to prove the following.
\begin{lemma}\label{lem:typical_props}
	If $\mathcal{K}$ is a typical mapping cone then at least one of the following conditions holds:
	\begin{enumerate}[(1)]
		\item $\mathcal{K}\subseteq\Pk{2}\circ t$, or
		\item $\SPk{2}\subseteq\mathcal{K}$.
	\end{enumerate}
\begin{proof}
	For the cones $\left\{\Pk{k},\SPk{k},\CP,\Pk{k}\circ t,\SPk{k}\circ t,\CP\circ t\right\}_{k=1}^n$, the assertion of the lemma is clearly true. It suffices to note that the operations $\cap$ and $\vee$ preserve the disjunction of conditions $(1)$ and $(2)$.
\end{proof}
\end{lemma}

It follows from comparing Lemmas~\ref{lem:mphi_props} and~\ref{lem:typical_props} that the mapping cone $\cl{M}_{\{\Phi\}}$ is untypical, which proves Theorem~\ref{thm:untypical_from_non_atomic}.

\subsection{Untypical Mapping Cones Within 2-Superpositives}\label{sec:untypical_s2_s2t}

In Section~\ref{sec:atomic_map} we saw that there are many untypical mapping cones that are not contained within $\cl{P}_2 \cup (\cl{P}_2 \circ t)$. In Section~\ref{sec:non_atomic_map} we then saw untypical mapping cones contained within $\SPk{2}$ but not contained in $\cl{P}_2 \circ t$. We now demonstrate that there are also many untypical mapping cones contained in $\SPk{2} \cap (\SPk{2} \circ t)$. 
First, it will be useful to prove the following lemma, similar to Lemma \ref{lem:typical_props} above.

\begin{lemma}\label{lem:SP2SP2t}
If $\mathcal{K}$ is a typical mapping cone, it must satisfy one of the following two properties
\begin{enumerate}[(1)]
\item $\mathcal{K}=\SPk{1}$
\item $\SPk{2}\cap\SPk{2}\circ t\subset\mathcal{K}$
\end{enumerate}
\begin{proof}
For the cones $\left\{\Pk{k},\SPk{k},\CP,\Pk{k}\circ t,\SPk{k}\circ t,\CP\circ t\right\}_{k=1}^n$, the assertion of the lemma is clearly true. Similarly as in Lemma \ref{lem:typical_props}, the operations $\cap$ and $\vee$ preserve the disjunction of conditions $(1)$ and $(2)$.
\end{proof}
\end{lemma}
To show that there exists a large family of untypical mapping cones contained in $\SPk{2}\cap\SPk{2}\circ t$, let us recall the results of a recent paper \cite{S2011} by one of the authors. In the paper, a family of extreme PPT states of rank $4$ in $3\times 3$ systems was characterized by their $\textnormal{SL}\left(3,\mathbbm{C}\right)\otimes\textnormal{SL}\left(3,\mathbbm{C}\right)$ equivalence to projections onto orthogonal complements of orthonormal unextendible product bases (UPBs), cf. also \cite{Bennett99,LS2010}.
\begin{defn}[Unextendible product basis]\label{defn:UPB}
A set of product vectors $\left\{u_i\otimes v_i\right\}_{i=1,2,\ldots,k}\subset\mathbbm{C}^m\otimes\mathbbm{C}^n$, $k\leqslant mn$ is called an (orthogonal) unextendible product basis (UPB) if the vectors $\phi_i\otimes\psi_i$ are mutually orthogonal and there is no additional product vector, orthogonal to all of them. We call a UPB orthonormal if the vectors $\phi_i\otimes\psi_i$ are normalized. 
\end{defn}

In the following, we shall prove that the inverse Jamio{\l}kowski--Choi transforms of all the PPT states considered in \cite{S2011} generate untypical mapping cones. Here, by the Jamio{\l}kowski--Choi transformation we mean the map $J:\Phi\mapsto C_{\Phi}$, where $C_{\Phi}$ is the Choi matrix of $\Phi$. We recall \cite{S2011} that an entangled PPT state $\rho$ of rank four in $3\times 3$ systems must have six product vectors $\left\{\phi_i\otimes\psi_i\right\}_{i=1}^6$ in its kernel. It was proved in \cite{S2011} that some quintuple, chosen from the six product vectors, span ${\rm ker}(\rho)$ and thus they must form a general unextendible product basis (gUPB), i.e. there is no product vector in their orthogonal complement. Without loss of generality, we may assume that the five product vectors are $\phi_i\otimes\psi_i$ for $i=1,2,\ldots,5$. The gUPB condition for $\left\{\phi_i\otimes\psi_i\right\}_{i=1}^5$ can be reformulated by saying that in $\left\{\phi_i\right\}_{i=1}^5$ and $\left\{\psi_i\right\}_{i=1}^5$, every triple of vectors is linearly independent. For our purposes, it will be useful to prove that, in fact, any quintuple of vectors in $\left\{\phi_i\otimes\psi_i\right\}_{i=1}^6$ actually form a gUPB. First of all, any basis of ${\rm ker}(\rho)$, consisting of product vectors, must be a gUPB, as the range of $\rho$, equal to ${\rm ker}(\rho)^{\bot}$, cannot contain a product vector, cf. Corollary 3.12 in \cite{S2011}. Since we assumed that the vectors $\left\{\phi_i\otimes\psi_i\right\}_{i=1}^5$ span ${\rm ker}(\rho)$ and $\phi_6\otimes\psi_6\in{\rm ker}(\rho)$, we must have
\begin{equation}
\phi_6\otimes\psi_6=\sum_{i=1}^5\lambda_i\phi_i\otimes\psi_i
\end{equation}
for some $\lambda_i\in\mathbbm{C}$, $i=1,2,\ldots,5$. Actually, we can prove that $\lambda_i\neq 0$ for all $i$. In \cite{S2011}, it was shown that $\lambda_4\neq 0$, for which the assumption that $\left\{\phi_i\otimes\psi_i\right\}_{i=1}^5$ is a gUPB, as well as Lemma 3.3 of \cite{S2011} was used. Neither the gUPB property nor the lemma depend on the ordering of the vectors $\phi_i\otimes\psi_i$, $i=1,2,\ldots,5$. Hence $\lambda_4\neq 0$ implies that $\lambda_i\neq 0$ for all $i=1,2,\ldots,5$. As a consequence, each of five-element subsets of $\left\{\phi_i\otimes\psi_i\right\}_{i=1}^6$ must span $\textnormal{ker}\rho$. Therefore each of them has to be a gUPB.

Let $\left\{\phi_{i_j}\otimes\psi_{i_j}\right\}_{j=1}^5$ be five arbitrary product vectors in the kernel of $\rho$. By repeating the argument of \cite[Sections E and F]{S2011}, one can show that there exists precisely one $\textnormal{SL}\left(3,\mathbbm{C}\right)\otimes\textnormal{SL}\left(3,\mathbbm{C}\right)$ transformation $\rho\mapsto\textnormal{Ad}_{A\otimes B}(\rho)=\left(A\otimes B\right)^{\ast}\rho\left(A\otimes B\right)$ that brings $\rho$ to the form
\begin{equation}
\chi\left(I-\sum_{j=1}^5\left|\tilde\phi_j\otimes\tilde\psi_j\right>\left<\tilde\phi_j\otimes\tilde\psi_j\right|\right),
\end{equation}
where $\chi>0$ and the vectors
\begin{equation}
\tilde\phi_j\otimes\tilde\psi_j=\chi_j\left(A^{-1}\phi_{i_j}\right)\otimes\left(B^{-1}\psi_{i_j}\right),
\end{equation}
$\chi_j\in\mathbbm{C}\forall_j$, form an orthonormal UPB. This leads us to the following

\begin{prop}\label{prop:onlysix}
Let $\rho$ be an entangled PPT state of rank $4$ in a $3\times 3$ system. There exist precisely six transformations 
\begin{equation}\label{eq:transformAdAB}
\rho\mapsto\textnormal{Ad}_{A\otimes B}(\rho)
\end{equation}
$A,B\in\textnormal{SL}\left(3,\mathbbm{C}\right)$ that bring $\rho$ to the form
\begin{equation}\label{eq:formUPB8}
\chi\left(I-\sum_{j=1}^5\left|\tilde\phi_j\otimes\tilde\psi_j\right>\left<\tilde\phi_j\otimes\tilde\psi_j\right|\right),
\end{equation} 
$\chi>0$, where $\left\{\tilde\phi_j\otimes\tilde\psi_j\right\}_{j=1}^5$ is an orthonormal UPB.
\begin{proof}
Let us choose $A,B\in\textnormal{SL}\left(3,\mathbbm{C}\right)$ such that
$\textnormal{Ad}_{A\otimes B}(\rho)$ is of the form \eqref{eq:formUPB8}. Clearly, the vectors $\left(A\otimes B\right)\left(\tilde\phi_j\otimes\tilde\psi_j\right)$ for $j=1,2,\ldots,5$ belong to ${\rm ker}(\rho)$, so that
\begin{equation}
\left\{\left(A\otimes B\right)\left(\tilde\phi_j\otimes\tilde\psi_j\right)\right\}_{j=1}^5\subset\left\{\chi'_i\phi_i\otimes\psi_i\right\}_{i=1}^6
\end{equation}
for some $\chi'_i\in\mathbbm{C}$. Thus, we must have
\begin{equation}
\tilde\phi_j\otimes\tilde\psi_j=\chi_j\left(A^{-1}\phi_{i_j}\right)\otimes\left(B^{-1}\psi_{i_j}\right),\quad i=1,2,\ldots,5
\end{equation}
for some five-element subset $\left\{\phi_{i_j}\otimes\psi_{i_j}\right\}_{j=1}^5$ of $\left\{\phi_i\otimes\psi_i\right\}_{i=1}^6$ and some $\chi_j\in\mathbbm{C}$.

As we remarked above, there exist presicely one suitable transformation $\rho\mapsto\textnormal{Ad}_{A\otimes B}(\rho)$, $A,B\in\textnormal{SL}\left(3,\mathbbm{C}\right)$ for each choice of $\left\{\phi_{i_j}\otimes\psi_{i_j}\right\}_{j=1}^5$. This gives us at most six different transformations that bring $\rho$ to the form \eqref{eq:formUPB8}. We need to show that no pair of them coincide. Without loss of generality, we may confine our discussion to transformations $A^{-1}\otimes B^{-1}$ that bring $\left\{\phi_i\otimes\psi_i\right\}_{i=1}^5$ and $\left\{\phi_j\otimes\psi_j\right\}_{j=2}^6$ to the UPB form. Assume for the moment that a single $A^{-1}\otimes B^{-1}$ does tha job for both $\left\{\phi_i\otimes\psi_i\right\}_{i=1}^5$ and $\left\{\phi_j\otimes\psi_j\right\}_{j=2}^6$. By the argument of \cite[Section F]{S2011}, we must require $\textnormal{Ad}_{A\otimes B}$ to bring $\rho$ to the form \eqref{eq:formUPB8}, where $\tilde\phi_j\otimes\tilde\psi_j$ is a transform of $\left\{\phi_i\otimes\psi_i\right\}_{i=1}^5$, and similarly for $\left\{\phi_j\otimes\psi_j\right\}_{j=2}^6$. Thus, we must have
\begin{multline}
\chi\left(I-\sum_{j=1}^5\left|\chi_i\right|^2\left|A^{-1}\phi_i\otimes B^{-1}\psi_i\right>\left<A^{-1}\phi_i\otimes B^{-1}\psi_i\right|\right)=\\=\chi'\left(I-\sum_{k=2}^6\left|\chi_k\right|^2\left|A^{-1}\phi_k\otimes B^{-1}\psi_k\right>\left<A^{-1}\phi_k\otimes B^{-1}\psi_k\right|\right)
\end{multline}
for some $\chi,\chi'>0$. Consequently, the subspaces spanned by $\left\{A^{-1}\phi_i\otimes B^{-1}\psi_i\right\}_{i=1}^5$ and $\left\{A^{-1}\phi_j\otimes B^{-1}\psi_j\right\}_{j=2}^6$ must be identical. The vectors $A^{-1}\phi_1\otimes B^{-1}\psi_1$ and $A^{-1}\phi_6\otimes B^{-1}\psi_6$ are both orthogonal to $\left\{A^{-1}\phi_k\otimes B^{-1}\psi_k\right\}_{k=2}^5$. Since $A^{-1}\phi_1\otimes B^{-1}\psi_1\neq A^{-1}\phi_6\otimes B^{-1}\psi_6$, there exists a vector $w$, orthogonal to $A^{-1}\phi_1\otimes B^{-1}\psi_1$, in the linear span of $A^{-1}\phi_1\otimes B^{-1}\psi_1$ and $A^{-1}\phi_6\otimes B^{-1}\psi_6$. We see that $w$ is orthogonal to all $A^{-1}\phi_i\otimes B^{-1}\psi_i$ for $i=1,2,\ldots,5$, but it is contained in $\left\{A^{-1}\phi_i\otimes B^{-1}\psi_i\right\}_{i=1}^6$. Hence, the sets $\left\{A^{-1}\phi_i\otimes B^{-1}\psi_i\right\}_{i=1}^5$ and $\left\{A^{-1}\phi_j\otimes B^{-1}\psi_j\right\}_{j=2}^6$ must span different subspaces, which is a contradiction.
\end{proof}
\end{prop}

A simple consequence of the above result is the following.

\begin{prop}\label{prop:onlyfive}
Let $\rho$ be an entangled $3\times3$ (unnormalized) quantum state of the form
\begin{equation}\label{eq:formUPB}
I-\sum_{i=1}^5\left|u_i\otimes v_i\right>\left<u_i\otimes v_i\right|,
\end{equation}
where $\left\{u_i\otimes v_i\right\}_{i=1}^5$ is an orthonormal UPB in $\mathbbm{C}^3\otimes\mathbbm{C}^3$. There exist at most five other states of the form \eqref{eq:formUPB}, equivalent to $\rho$ via transformations $\rho\mapsto\textnormal{Ad}_{A\otimes B}(\rho)=\left(A\otimes B\right)^{\ast}\rho\left(A\otimes B\right)$, where $A,B\in\textnormal{GL}\left(3,\mathbbm{C}\right)$.
\begin{proof}
Since $\rho$ is a PPT entangled state of rank four in a $3\times 3$ system, by Proposition \ref{prop:onlysix}, there exist precisely six transformations $\rho\mapsto\textnormal{Ad}_{A\otimes B}(\rho)$ with $A,B\in\textnormal{SL}\left(3,\mathbbm{C}\right)$ that bring $\rho$ to the form \eqref{eq:formUPB}, multiplied by some $\chi>0$. Since $\rho$ is itself of the form \eqref{eq:formUPB}, there exist at most five other states of the form $\chi\left(I-\sum_{i=1}^5\left|u_i\otimes v_i\right>\left<u_i\otimes v_i\right|\right)$, equivalent to $\rho$ via transformations of the type $\rho\mapsto\textnormal{Ad}_{A\otimes B}(\rho)$, $A,B\in\textnormal{SL}\left(3,\mathbbm{C}\right)$. From this, it immediately follows that there can only exist five states of the form \eqref{eq:formUPB}, equivalent to $\rho$ via transformations $\rho\mapsto\textnormal{Ad}_{A\otimes B}(\rho)$ with $A,B\in\textnormal{GL}\left(3,\mathbbm{C}\right)$.
\end{proof}
\end{prop}

We are now ready to prove the main result of the present section.

\begin{thm}\label{thm:untypical3x3}
Let $\Phi$ denote the inverse of the Jamio{\l}kowski--Choi transform a $3\times 3$ (unnormalized) state
\begin{equation}
\rho=C_{\Phi}=I-\sum_{i=1}^5\left|u_i\otimes v_i\right>\left<u_i\otimes v_i\right|,
\end{equation}
where $\left\{u_i\otimes v_i\right\}_{i=1}^5$ is an orthonormal unextendible product basis in $\mathbbm{C}^3\otimes\mathbbm{C}^3$. Then the mapping cone $\mathcal{M}_{\left\{\Phi\right\}}$ generated by $\Phi$ is untypical.
\begin{proof}
By Remark 3.22 of \cite{S2011}, we know that $\rho$ and $(id\otimes t)(\rho)$ are of Schmidt rank $2$, hence $\Phi\in\SPk{2}\cap\SPk{2}\circ t$, cf. e.g. \cite{SSZ09}. We also have $\Phi\not\in\SPk{1}$ as a consequence of non-separability of $\rho$. If we prove that $\SPk{2}\cap\SPk{2}\circ t\not\subset\mathcal{M}_{\left\{\Phi\right\}}$, Lemma \ref{lem:SP2SP2t} will tell us that $\mathcal{M}_{\left\{\Phi\right\}}$ is not a typical mapping cone. Let us denote by $J$ the Jamio{\l}kowski--Choi map, i.e. $J:\Phi\mapsto C_{\Phi}$. If $\SPk{2}\cap\SPk{2}\circ t$ was contained in $\mathcal{M}_{\left\{\Phi\right\}}$, all the maps of the form
\begin{equation}\label{eq:Jminusone}
\tilde\Phi=J^{-1}\left(I-\sum_{j=1}^5\left|\tilde v_j\otimes\tilde w_j\right>\left<\tilde v_j\otimes\tilde w_j\right|\right)
\end{equation}
for some orthonormal UPB $\left\{\tilde v_j\otimes\tilde w_j\right\}_{j=1}^5$ would have to be contained in $\mathcal{M}_{\left\{\Phi\right\}}$. Due to the extremality of $\tilde\Phi$, cf. \cite{LS2010}, we would need to have 
\begin{equation}
\tilde\Phi=\textnormal{Ad}_A\circ\Phi\circ\textnormal{Ad}_B
\end{equation}
for some $A,B\in\textnormal{GL}\left(3,\mathbbm{C}\right)$, which translates to
\begin{equation}
I-\sum_{i=1}^5\left|\tilde v_j\otimes\tilde w_j\right>\left<\tilde v_j\otimes\tilde w_j\right|=\textnormal{Ad}_{A\otimes B}\left(I-\sum_{i=1}\left|v_i\otimes w_i\right>\left<v_i\otimes w_i\right|\right),
\end{equation}
where both $\left\{v_i\otimes w_i\right\}_{i=1}^5$ and $\left\{\tilde v_j\otimes\tilde w_j\right\}_{j=1}^5$ are orthonormal unextendible product bases. However, we know from Proposition \ref{prop:onlyfive} that the above equality can hold for at most five choices of the unextendible product basis $\left\{\tilde v_j\otimes\tilde w_j\right\}_{j=1}^5$, different than $\left\{v_i\otimes w_i\right\}_{i=1}^5$. For other choices of $\left\{\tilde v_j\otimes\tilde w_j\right\}_{j=1}^5$, the equality cannot hold. Consequently, there exist maps $\tilde\Phi$ of the form \eqref{eq:Jminusone} that are not elements of $\mathcal{M}_{\left\{\Phi\right\}}$. Since $\tilde\Phi\in\SPk{2}\cap\SPk{2}\circ t$, we obtain $\SPk{2}\cap\SPk{2}\circ t\not\subset\mathcal{M}_{\left\{\Phi\right\}}$, which implies that $\mathcal{M}_{\left\{\Phi\right\}}$ is an untypical mapping cone, by Lemma \ref{lem:SP2SP2t}.
\end{proof}
\end{thm}




\subsection{Typical and Untypical Mapping Cones in the \texorpdfstring{$n = 2$}{n = 2} Case}\label{sec:n2untypical}

The examples in the previous sections demonstrate some different untypical mapping cones when the dimension $n \geq 3$. However, we have not yet seen any untypical convex mapping cones in the $n = 2$ case. We now characterize convex mapping cones generated by single maps in the $n = 2$ case and show that, while there are only four different typical convex mapping cones in this setting, there are many untypical mapping cones.
\begin{prop}\label{prop:four_typical_cones}
	The only four typical convex mapping cones in $\cl{P}_1(\cl{L}(\cl{H}_2))$ are $\cl{P}_1$, $\cl{CP}$, $\cl{CP} \circ t$, and $\SPk{1}$.
\end{prop}
\begin{proof}
	It is well-known that $\cl{CP} \vee (\cl{CP} \circ t) = \cl{P}_1$ and $\cl{CP} \cap (\cl{CP} \circ t) = \SPk{1}$ in this case \cite{S63,W76}, so the result is trivial.
\end{proof}

We now consider how we might construct an untypical mapping cone from a single map, in the same manner as was done in Sections~\ref{sec:atomic_map} and~\ref{sec:non_atomic_map}.
\begin{thm}\label{thm:map_cones_n2}
	Let $\Phi : \cl{L}(\cl{H}_2) \rightarrow \cl{L}(\cl{H}_2)$ be a positive map. The following characterizes when $\cl{M}_{\{\Phi\}}$ is typical or untypical:
	\begin{enumerate}[(a)]
		\item If $\Phi \in \SPk{1}$, then $\cl{M}_{\{\Phi\}} = \SPk{1}$.
		\item If $\Phi \in \cl{CP} \backslash \SPk{1}$ and ${\rm rank}(C_\Phi) = 1$, then $\cl{M}_{\{\Phi\}} = \cl{CP}$.
		\item If $\Phi \in \cl{CP} \backslash \SPk{1}$ and ${\rm rank}(C_\Phi) \geq 2$, then $\cl{M}_{\{\Phi\}}$ is untypical.
		\item If $\Phi \in (\cl{CP} \circ t) \backslash \SPk{1}$ and ${\rm rank}(C_{\Phi\circ t}) = 1$, then $\cl{M}_{\{\Phi\}} = \cl{CP} \circ t$.
		\item If $\Phi \in (\cl{CP} \circ t) \backslash \SPk{1}$ and ${\rm rank}(C_{\Phi\circ t}) \geq 2$, then $\cl{M}_{\{\Phi\}}$ is untypical.
		\item If $\Phi \notin \cl{CP} \cup (\cl{CP} \circ t)$, then $\cl{M}_{\{\Phi\}}$ is untypical.
	\end{enumerate}
\end{thm}
\begin{proof}
	 Condition (a) is true (regardless of $n$) by Proposition~\ref{prop:sing_map_cone_superpos}.
	 
	 For (b), note that if $\Phi \in \cl{CP} \backslash \SPk{1}$ and ${\rm rank}(C_\Phi) = 1$, then $\Phi = {\rm Ad}_V$ for some $V$ with ${\rm rank}(V) = 2$. Thus we can write $V = \mathbf{a}_1\mathbf{b}_1^* + \mathbf{a}_2\mathbf{b}_2^*$ for some linearly independent sets $\{\mathbf{a}_1,\mathbf{a}_2\}$ and $\{\mathbf{b}_1,\mathbf{b}_2\}$. Define $A,B \in \cl{L}(\cl{H}_2)$ to be operators that satisfy $A\mathbf{a}_i = \mathbf{e}_i$ and $B\mathbf{b}_i = \mathbf{e}_i$ for $i = 1,2$. Then ${\rm Ad}_A \circ \Phi \circ {\rm Ad}_B = id$. Thus $id \in \cl{M}_{\{\Phi\}}$, so $\cl{CP} \subseteq \cl{M}_{\{\Phi\}}$. The other inclusion is trivial.

	 For (c), note that $\SPk{1} \subsetneq \cl{M}_{\{\Phi\}} \subseteq \cl{CP}$, so the only way that $\cl{M}_{\{\Phi\}}$ could be typical is if $\cl{M}_{\{\Phi\}} = \cl{CP}$. We will rule out this possibility by showing that $id \notin \cl{M}_{\{\Phi\}}$. To this end, write $\Phi = \sum_i {\rm Ad}_{A_i}$ with ${\rm rank}(A_1) = 2$. Since $id$ is an extreme point of the set of positive maps, we suppose (in order to get a contradiction) that there exist $A,B \in \cl{L}(\cl{H}_2)$ such that ${\rm Ad}_A \circ \Phi \circ {\rm Ad}_B = id$. If either ${\rm rank}(A) = 1$ or ${\rm rank}(B) = 1$ then ${\rm Ad}_A \circ \Phi \circ {\rm Ad}_B$ would be superpositive and thus not equal to $id$. However, if ${\rm rank}(A) = {\rm rank}(B) = 2$ then ${\rm rank}(C_{{\rm Ad}_A \circ \Phi \circ {\rm Ad}_B}) \geq 2$, so ${\rm Ad}_A \circ \Phi \circ {\rm Ad}_B \neq id$ in this case as well.

	 Conditions (d) and (e) follow immediately from conditions (b) and (c) by replacing $\Phi$ with $\Phi \circ t$.
	 
	 To see (f), note that if $\cl{M}_{\{\Phi\}}$ were typical, it would have to equal $\cl{P}_1$ since $\Phi$ is not a member of any of the other three typical convex mapping cones. In particular, this would imply that $id, t \in \cl{M}_{\{\Phi\}}$. Since $id$ is an extreme point of the set of positive maps, there must exist $A,B \in \cl{L}(\cl{H}_2)$ such that ${\rm Ad}_A \circ \Phi \circ {\rm Ad}_B = id$. As in the proof of Theorem~\ref{thm:atomic_untypical}, $A$ and $B$ must be invertible because $id$ has full rank as a linear operator, so $\Phi = {\rm Ad}_{A^{-1}} \circ \Phi \circ {\rm Ad}_{B^{-1}}$. This implies that $\Phi$ is completely positive, which is a contradiction and completes the proof. Note that because $t$ is also extreme in the set of positive maps, we could have similarly gotten a contradiction by showing that if $t \in \cl{M}_{\{\Phi\}}$ then $\Phi \in \cl{CP} \circ t$.
\end{proof}

\section{Spin Factors as Mapping Cones}\label{sec:spin_factors}

A construction of spin factors follows. Let
\begin{align*}
	\sigma_1 := \begin{bmatrix}1 & 0 \\ 0 & -1\end{bmatrix}, \quad \sigma_2 := \begin{bmatrix}0 & 1 \\ 1 & 0\end{bmatrix}, \quad \sigma_3 := \begin{bmatrix}0 & -i \\ i & 0\end{bmatrix}
\end{align*}
be the usual Pauli spin matrices on $\bb{C}^2$. Let $\sigma_3^{\otimes n}$ be the $n$-fold tensor product of $\sigma_3$ with itself and define
\begin{align*}
	s_1 & := \sigma_1 \otimes \overbrace{I \otimes I \otimes \cdots \otimes I}^{\text{$n-1$ times}}, \\
	s_2 & := \sigma_2 \otimes I \otimes I \otimes \cdots \otimes I, \\
	s_3 & := \sigma_3 \otimes \sigma_1 \otimes I \otimes \cdots \otimes I, \\
	s_4 & := \sigma_3 \otimes \sigma_2 \otimes I \otimes \cdots \otimes I, \\
	& \ \ \ \ \ \ \ \vdots \\
	s_{2n-1} & := \sigma_3^{n-1} \otimes \sigma_1 \\
	s_{2n} & := \sigma_3^{n-1} \otimes \sigma_2.
\end{align*}

Then the real linear span $V_k := {\rm span}\{I, s_1, s_2, \ldots, s_k\}$ (for $1 \leq k \leq 2n$) is a \emph{spin factor} of dimension $k+1$ (our decision to use the first $k$ $s_i$'s is merely for convenience, as spin factors obtained by choosing a different set of $k$ of these operators are isomorphic). We can regard $V_{2k-1}$ and $V_{2k}$ as contained in $\cl{L}(\bb{C}^{2^k})$ in the obvious way -- see \cite{HS84} for further details. If $\cl{H} = \bb{C}^{2^k} \otimes \cl{H}_m$ then the span of the set $\{I \otimes I_m, s_1 \otimes I_m, s_2 \otimes I_m, \ldots, s_k \otimes I_m\}$ is also a spin factor, which we denote $V_k \otimes I_m$.

For the remainder of this section, we use $\tau$ to denote the tracial state -- i.e., the trace functional normalized so that $\tau(I) = 1$. The positive projection $E_k$ onto $V_k$ is given by \cite{ES79}
\begin{align*}
	E_k(X) = \sum_{i=0}^k \tau(s_i X)s_i, \quad s_0 := I,
\end{align*}
and the positive projection $F_k$ onto $V_k \otimes I_m$ is given by
\begin{align*}
	F_k(A \otimes B) & = \sum_{i=0}^k \tau\big((s_i \otimes I_m)(A \otimes B)\big)s_i \otimes I_m \\
	& = \sum_{i=0}^k \tau(s_i A)s_i \otimes \tau(B)I_m \\
	& = (E_k \otimes \tau^\prime)(A \otimes B),
\end{align*}
where $\tau^\prime(X) = \tau(X)I_m$. Thus $F_k = E_k \otimes \tau^\prime$.

\subsection{Analogy with k-Positive and k-Superpositive Maps}\label{sec:spin_factor_analogy}

We now present some results that help us discuss the mapping cones generated by the projections $E_k$ and $F_k$ onto the spin factors $V_k$ and $V_k \otimes I_m$, respectively. We will see that there is a natural analogy between the cones of $k$-superpositive maps and these cones, and similarly their dual cones are analogous in a natural way to the cones of $k$-positive maps.
\begin{prop}\label{prop:dual_atomic} 
	If $\Phi = \Phi^\dagger = t \circ \Phi \circ t \in \cl{P}_1$ is atomic then $\cl{M}_{\{\Phi\}}^\circ = \{ \Psi \in \cl{P}_1 : \Psi \otimes \Phi \text{ is positive}\}$ is untypical.
\end{prop}
\begin{proof}
	The formula for $\cl{M}_{\{\Phi\}}^\circ$ is given by \cite{Sto11}. Since the dual cone of any typical mapping cone is again typical, the result follows from Theorem~\ref{thm:atomic_untypical}.
\end{proof}

We expect that the following lemma is well-known, but we have not been able to find a reference for it.
\begin{lemma}\label{lem:spin_plf} 
	Let $\Phi : \cl{L}(\cl{H}_n) \rightarrow \cl{L}(\cl{H}_n)$ satisfy $\Phi = \Phi^\dagger = t \circ \Phi \circ t$. Let $\cl{H} = \cl{H}_n \otimes \cl{H}_m$ and let $\psi : \cl{L}(\cl{H}_m) \rightarrow \bb{C}$ be a positive linear functional with $\psi(I_m) = 1$. Then $\Phi$ is positive if and only if $\Phi \otimes \psi$ is positive.
\end{lemma}
\begin{proof}
	For the ``if'' direction, let $0 \leq \rho \in \cl{L}(\cl{H}_n)$. Then $0 \leq (\Phi \otimes \psi)(\rho \otimes I_m) = \Phi(\rho) \otimes 1 = \Phi(\rho)$, so $\Phi$ is positive.
	
	For the ``only if'' direction, let $\psi_i$ ($i = 1,2$) be positive linear functionals on $\cl{L}(\cl{H}_m)$ with $\psi_i(I_m) = 1$. Define $\omega := \psi_1 \otimes \psi_2$. Since $\cl{P}_1(\cl{H})$ is the dual cone of $\cl{SP}_1(\cl{H})$, it follows from \cite{Sto11} that $\Phi \otimes \omega$ is positive. Thus $\Phi \otimes \psi_1 \otimes \psi_2$ is positive, so $\Phi \otimes \psi_1$ is positive by the first part of this proof.
\end{proof}

By Lemma~\ref{lem:spin_plf}, if $\Phi \in \cl{P}_1(\cl{H})$ then
\begin{align}\label{eq:lem_3_proj}
	\Phi \otimes E_k \text{ is positive } \ \Longleftrightarrow \ \Phi \otimes F_k \text{ is positive},
\end{align}
where $E_k$ and $F_k$ are the projections onto the spin factors $V_k$ and $V_k \otimes I_m$, respectively.
\begin{thm}\label{thm:spin_anal_kpos} 
	Let $\cl{H} := \bb{C}^{2^n} \otimes \cl{H}_m$ and $\Phi \in \cl{P}_1(\cl{H})$. Then $\Phi \in \cl{M}_{F_k}^\circ$ if and only if $\Phi \otimes E_k$ is positive.
\end{thm}
\begin{proof}
	This result follows immediately from \cite[Corollary~4]{Sto11} and Equation~\eqref{eq:lem_3_proj}.
\end{proof}

Theorem~\ref{thm:spin_anal_kpos} provides a natural analogy between the cones of $k$-positive maps and the mapping cones $\cl{M}_{F_k}^\circ$ (and by duality, a natural analogy between the cones of $k$-superpositive maps and the mapping cones $\cl{M}_{F_k}$). Indeed, $\Phi \otimes id_k$ is positive if and only if $\Phi$ is $k$-positive whereas Theorem~\ref{thm:spin_anal_kpos} says that $\Phi \otimes E_k$ is positive if and only if $\Phi \in \cl{M}_{F_k}^\circ$. One important distinction between these two cases, however, arises from the fact that $E_k$ (and hence $F_k$) is atomic for $k \neq 2, 3, 5$ \cite{Rob85}. It follows from Theorem~\ref{thm:atomic_untypical} that the cones $\cl{M}_{F_{k}}$ (which are analogous to the $k$-superpositive maps) and $\cl{M}_{F_{k}}^\circ$ (which are analogous to the $k$-positive maps) are untypical when $k \neq 2, 3, 5$.

Also, much like we have $\SPk{1} \subset \SPk{2} \subset \cdots$, we have a natural family of inclusions in this setting as well. Because we have $F_k = PF_{k+1}$, where $P$ is the conditional projection of the $C^*$-algebra generated by $V_{k+1}$ onto that generated by $V_k$, it follows that $\cl{M}_{F_{1}} \subset \cl{M}_{F_{2}} \subset \cdots$.

\section{The Partial Order Induced by Mapping Cones}\label{sec:partial_order}

The method of generating mapping cones of Section~\ref{sec:cone_generated} leads to a natural partial order on positive maps (or even sets of positive maps). We write $\Phi \succeq \Psi$ if $\cl{M}_{\{\Phi\}} \supseteq \cl{M}_{\{\Psi\}}$, and we note that ``$\succeq$'' is a partial order on the set of positive maps if we identify $\Phi$ with $\Psi$ whenever $\cl{M}_{\{\Phi\}} = \cl{M}_{\{\Psi\}}$. We write $\Phi \approx \Psi$ in this case. It would be interesting to have an alternate characterization of when $\Phi \approx \Psi$, but we leave this as an open problem.

This partial order has a natural interpretation in quantum information theory. If $\Phi \succeq \Psi$, then we can think of $\Phi$ as being better at detecting entanglement than $\Psi$. Indeed, if $\Phi \succeq \Psi$ then there exist $\{A_i\}$ and $\{B_i\}$ such that $\Psi = \sum_i {\rm Ad}_{A_i} \circ \Phi \circ {\rm Ad}_{B_i}$. It is then easily-verified that if $X \geq 0$ then $(id_n \otimes \Psi)(X) \ngeq 0$ implies that $(id_n \otimes \Phi)(X) \ngeq 0$. Using the terminology of \cite{LKCH00}, this means that $\Phi$ is ``finer'' than $\Psi$, which means that any entanglement detected by $\Psi$ is also detected by $\Phi$.

However, the reverse implication does not hold -- there are positive maps $\Phi$, $\Psi$ such that $\Phi$ is finer than $\Psi$, but $\Phi \nsucceq \Psi$. To see this, recall from \cite{LKCH00} that $\Phi$ is finer than $\Psi$ if and only if there is a completely positive map $\Omega$ such that $\Psi = \Phi + \Omega$. Let $n \geq 3$ and let $P_1$ and $P_2$ be nonzero orthogonal projections in $\cl{L}(\cl{H}_n)$ such that ${\rm rank}(P_1) \geq 2$, with sum $P_1+P_2=I$. Let $\Phi : \cl{L}(\cl{H}_n) \rightarrow \cl{L}(\cl{H}_n)$ be a positive (but not completely positive) map satisfying $\Phi = {\rm Ad}_{P_1} \circ \Phi \circ {\rm Ad}_{P_1}$ and define $\Omega := {\rm Ad}_{P_2}$. If $A \in \cl{L}(\cl{H}_n)$, then ${\rm Ad}_A \circ \Phi = {\rm Ad}_A \circ {\rm Ad}_{P_1} \circ \Phi$ and similarly for $\Phi \circ {\rm Ad}_A$, so it follows that $\cl{M}_{\{\Phi\}} \subseteq \cl{M}_{\{\Phi+\Omega\}}$. Similar arguments show that $\cl{M}_{\{\Omega\}} \subseteq \cl{M}_{\{\Phi+\Omega\}}$, so $\cl{M}_{\{\Phi+\Omega\}} = \cl{M}_{\{\Phi\}} \vee \cl{M}_{\{\Omega\}}$. It follows that $\cl{M}_{\{\Phi+\Omega\}} \supsetneq \cl{M}_{\{\Phi\}}$, so $\Phi \nsucceq (\Phi + \Omega)$ even though $\Phi$ is finer than $\Phi + \Omega$.

The reverse implication also does not hold in the $n = 2$ case. Given any map $\Psi \in \cl{P}_1 \backslash (\cl{CP} \circ t)$ in this case, there is always a finer positive map $\Phi \in \cl{CP} \circ t$ (since we can write $\Psi = \Omega + \Phi$ for some $\Omega \in \cl{CP}$, $\Phi \in \cl{CP} \circ t$). However, there clearly is no map $\Phi \in \cl{CP} \circ t$ with $\Phi \succeq \Psi$, since $\Psi \notin \cl{CP} \circ t \supseteq \cl{M}_{\{\Phi\}}$.

As an example to illustrate the relationship between this partial order and entanglement detection, recall the reduction map $R(X) = \Tr(X)I - X$. It was shown in Proposition~\ref{prop:sing_map_cone_reduc} that $R \circ t \in \SPk{2}$ (so in particular, $R \circ t$ is completely positive). It follows that $\Phi \circ t = R$ for the completely positive map $\Phi := R \circ t$, so $t \succeq R$. This tells us that $t$ is more useful for entanglement detection than $R$, as any entanglement detected by $R$ must also be detected by $t$; a fact that is well-known \cite{HH99}.

\vspace{0.1in} \noindent{\bf Acknowledgements.} Nathaniel Johnston was supported by the University of Guelph Brock Scholarship. {\L}ukasz Skowronek acknowledges the support by the International PhD Projects Programme of the Foundation for Polish Science within the European Regional Development Fund of the European Union, agreement no. MPD/2009/6.

\bibliographystyle{alpha}
\bibliography{bib}

\newcommand{\etalchar}[1]{$^{#1}$}
\begin{thebibliography}{BDM{\etalchar{+}}99}

\bibitem[And04]{And04}
T.~Ando.
\newblock Cones and norms in the tensor product of matrix spaces.
\newblock {\em Linear Algebra Appl.}, 379:3--41, 2004.

\bibitem[BDM{\etalchar{+}}99]{Bennett99}
C.~H. Bennett, D.~P. DiVicenzo, T.~Mor, P.~W. Shor, J.~A. Smolin, and B.~M.
  Terhal.
\newblock Unextendible product bases and bound entanglement.
\newblock {\em Phys. Rev. Lett.}, 82:5385--5388, 1999.

\bibitem[BFP04]{BFP04b}
F.~Benatti, R.~Floreanini, and M.~Piani.
\newblock Quantum dynamical semigroups and non-decomposable positive maps.
\newblock {\em Phys. Lett. A}, 326:187--198, 2004.

\bibitem[CAG99]{CAG99}
N.~J. Cerf, C.~Adami, and R.~M. Gingrich.
\newblock Reduction criterion for separability.
\newblock {\em Phys. Rev. A}, 60:898--909, 1999.

\bibitem[CDKL01]{CDKL01}
J.~I. Cirac, W.~D\"{u}r, B.~Kraus, and M.~Lewenstein.
\newblock Entangling operations and their implementation using a small amount
  of entanglement.
\newblock {\em Phys. Rev. Lett.}, 86:544--547, 2001.

\bibitem[Cho75a]{C75}
M.-D. Choi.
\newblock Completely positive linear maps on complex matrices.
\newblock {\em Linear Algebra Appl.}, 10:285--290, 1975.

\bibitem[Cho75b]{Cho75}
M.-D. Choi.
\newblock Positive semidefinite biquadratic forms.
\newblock {\em Linear Algebra Appl.}, 12:95--100, 1975.

\bibitem[CK06]{CK06}
D.~Chru\'{s}ci\'{n}ski and A.~Kossakowski.
\newblock On partially entanglement breaking channels.
\newblock {\em Open Syst. Inf. Dyn.}, 13:17--26, 2006.

\bibitem[CKL92]{CKL92}
S.-J. Cho, S.-H. Kye, and S.~G. Lee.
\newblock Generalized {C}hoi maps in $3$-dimensional matrix algebras.
\newblock {\em Linear Algebra Appl.}, 171:213--224, 1992.

\bibitem[ES79]{ES79}
E.~Effros and E.~St{\o}rmer.
\newblock Positive projections and jordan structure in operator algebras.
\newblock {\em Math. Scand.}, 45:127--138, 1979.

\bibitem[Hal06]{Hal06}
W.~Hall.
\newblock A new criterion for indecomposability of positive maps.
\newblock {\em J. Phys. A: Math. Gen.}, 39:14119, 2006.

\bibitem[HH99]{HH99}
M.~Horodecki and P.~Horodecki.
\newblock Reduction criterion of separability and limits for a class of
  distillation protocols.
\newblock {\em Phys. Rev. A}, 59:4206--4216, 1999.

\bibitem[HHH96]{HHH96}
M.~Horodecki, P.~Horodecki, and R.~Horodecki.
\newblock Separability of mixed states: Necessary and sufficient conditions.
\newblock {\em Phys. Lett. A}, 223:1--8, 1996.

\bibitem[HHH98]{HHH98}
M.~Horodecki, P.~Horodecki, and R.~Horodecki.
\newblock Mixed-state entanglement and distillation: Is there a ``bound''
  entanglement in nature?
\newblock {\em Phys. Rev. Lett.}, 80:5239--5242, 1998.

\bibitem[HHHH09]{HHH09}
R.~Horodecki, P.~Horodecki, M.~Horodecki, and K.~Horodecki.
\newblock Quantum entanglement.
\newblock {\em Rev. Mod. Phys.}, 81:865--942, 2009.

\bibitem[HOS84]{HS84}
H.~Hanche-Olsen and E.~St{\o}rmer.
\newblock {\em Jordan operator algebras}, volume~21 of {\em Monographs and
  Studies in Mathematics}.
\newblock Pitman, 1984.

\bibitem[HSR03]{HSR03}
M.~Horodecki, P.~W. Shor, and M.~B. Ruskai.
\newblock General entanglement breaking channels.
\newblock {\em Rev. Math. Phys.}, 15:629--641, 2003.

\bibitem[Jam72]{J72}
A.~Jamio{\l}kowski.
\newblock Linear transformations which preserve trace and positive
  semidefiniteness of operators.
\newblock {\em Rep. Math. Phys.}, 3:275--278, 1972.

\bibitem[JS12]{JS11}
N.~Johnston and E.~St{\o}rmer.
\newblock Mapping cones are operator systems.
\newblock \emph{Bull. Lond. Math. Soc.}. doi:
  \href{http://dx.doi.org/10.1112/blms/bds006}{10.1112/blms/bds006}, 2012.

\bibitem[LKCH00]{LKCH00}
M.~Lewenstein, B.~Kraus, J.~I. Cirac, and P.~Horodecki.
\newblock Optimization of entanglement witnesses.
\newblock {\em Phys. Rev. A}, 62:052310, 2000.

\bibitem[LMS10]{LS2010}
J.~M. Leinaas, J.~Myrheim, and P.~{\O}. Sollid.
\newblock Low-rank extremal positive-partial-transpose states and unextendible
  product bases.
\newblock {\em Phys. Rev. A}, 81, 2010.

\bibitem[Mar10]{Mar10}
M.~Marciniak.
\newblock On extremal positive maps acting between type {I} factors.
\newblock In {\em Noncommutative Harmonic Analysis with Applications to
  Probability {II}}, pages 201--221, 2010.

\bibitem[Osa91]{Osa91}
H.~Osaka.
\newblock Indecomposable positive maps in low dimensional matrix algebras.
\newblock {\em Linear Algebra Appl.}, 153:73--83, 1991.

\bibitem[Per96]{P96}
A.~Peres.
\newblock Separability criterion for density matrices.
\newblock {\em Phys. Rev. Lett.}, 77:1413--1415, 1996.

\bibitem[Rai97]{Rai97}
E.~M. Rains.
\newblock Entanglement purification via separable superoperators.
\newblock E-print:
  \href{http://arxiv.org/abs/quant-ph/9707002}{arXiv:quant-ph/9707002}, 1997.

\bibitem[Rob85]{Rob85}
A.~G. Robertson.
\newblock Positive projections on {C}$^*$-algebras and an extremal positive
  map.
\newblock {\em J. London Math. Soc.}, 32:133--140, 1985.

\bibitem[Sko11a]{Sko11}
{\L}.~Skowronek.
\newblock Cones with a mapping cone symmetry in the finite-dimensional case.
\newblock {\em Linear Algebra Appl.}, 435:361--370, 2011.

\bibitem[Sko11b]{S2011}
{\L}.~Skowronek.
\newblock Three-by-three bound entanglement with general unextendible product
  bases.
\newblock {\em J. Math. Phys.}, 52:122202, 2011.

\bibitem[Sko12]{Sko12}
{\L}.~Skowronek.
\newblock There is no analogue of the positive partial transpose criterion in
  the three-by-three case.
\newblock Upcoming draft, 2012.

\bibitem[SS12]{SS10}
{\L}.~Skowronek and E.~St{\o}rmer.
\newblock Choi matrices, norms and entanglement associated with positive maps
  on matrix algebras.
\newblock {\em J. Funct. Anal.}, 262:639--647, 2012.

\bibitem[SS{\.Z}09]{SSZ09}
{\L}.~Skowronek, E.~St{\o}rmer, and K.~{\.Z}yczkowski.
\newblock Cones of positive maps and their duality relations.
\newblock {\em J. Math. Phys.}, 50:062106, 2009.

\bibitem[St{\o}63]{S63}
E.~St{\o}rmer.
\newblock Positive linear maps of operator algebras.
\newblock {\em Acta Math.}, 110:233--278, 1963.

\bibitem[St{\o}86]{S86}
E.~St{\o}rmer.
\newblock Extension of positive maps into ${B}({H})$.
\newblock {\em J. Funct. Anal.}, 66:235--254, 1986.

\bibitem[St{\o}11a]{S11}
E.~St{\o}rmer.
\newblock Mapping cones of positive maps.
\newblock {\em Math. Scand.}, 108:223--232, 2011.

\bibitem[St{\o}11b]{Sto11}
E.~St{\o}rmer.
\newblock Tensor products of positive maps of matrix algebras.
\newblock E-print: \href{http://arxiv.org/abs/1101.2114}{arXiv:1101.2114}
  [math.OA], 2011.

\bibitem[TH00]{TH00}
B.~M. Terhal and P.~Horodecki.
\newblock Schmidt number for density matrices.
\newblock {\em Phys. Rev. A}, 61:040301(R), 2000.

\bibitem[Tom85]{Tom85}
J.~Tomiyama.
\newblock On the geometry of positive maps in matrix algebras {II}.
\newblock {\em Linear Algebra Appl.}, 69:169--177, 1985.

\bibitem[TT88]{TT88}
K.~Tanahashi and J.~Tomiyama.
\newblock Indecomposable positive maps in matrix algebras.
\newblock {\em Canad. Math. Bull.}, 31:308--317, 1988.

\bibitem[Wor76]{W76}
S.~L. Woronowicz.
\newblock Positive maps of low dimensional matrix algebras.
\newblock {\em Rep. Math. Phys.}, 10:165--183, 1976.

\end{thebibliography}
\end{document}